\pdfoutput=1
\documentclass[a4paper]{amsart}
\usepackage[utf8]{inputenc}
\usepackage{amsmath}
\usepackage{amsfonts}
\usepackage{amssymb}
\usepackage{amsthm}
\usepackage[numbers]{natbib}
\usepackage{graphicx}
\usepackage{color}
\usepackage{mathtools}
\usepackage{tikz-cd}
    \usetikzlibrary{arrows}
    \usetikzlibrary{calc}
\usepackage{tcolorbox}
\usepackage{hyperref}

\hypersetup{
	linkcolor  = black,
	urlcolor   = blue,
        citecolor = black,
	colorlinks = true,
}  

\newtheorem{theorem}{Theorem}[section]
\newtheorem{lemma}[theorem]{Lemma}
\newtheorem{corollary}[theorem]{Corollary}

\theoremstyle{definition} 

\newtheorem{remark}[theorem]{Remark}
\newtheorem{definition}[theorem]{Definition}
\newtheorem{proposition}[theorem]{Proposition}

\newcommand{\Vol}{\mathrm{Vol}(M)}
\newcommand{\Diff}{\mathrm{Diff}(M)}
\newcommand{\DiffRn}{\mathrm{Diff}(\mathbb{R}^n)}
\newcommand{\DiffR}{\Diff\times \mathbb{R}_+}
\newcommand{\SDiff}{\mathrm{Diff}_\mu(M)}

\newcommand*{\G}[1][]{\mathcal{G}_{#1}} 
    \newcommand*{\bG}[1][]{\bar{\mathcal{G}}_{#1}}
\newcommand*{\Gfin}[1][]{\mathcal{G}^\text{fin}_{#1}} 
    \newcommand*{\bGfin}[1][]{\bar{\mathcal{G}}^\text{fin}_{#1}}
\newcommand*{\Gaff}[1][]{\mathcal{G}^\text{aff}_{#1}} 
    \newcommand*{\bGaff}[1][]{\bar{\mathcal{G}}^\text{aff}_{#1}}
\newcommand*{\Glarge}[1][]{\mathcal{G}^{\rm big}_{#1}} 
    \newcommand*{\bGlarge}[1][]{\bar{\mathcal{G}}^{\rm big}_{#1}}
\newcommand*{\Gdiv}[1][]{\mathcal{G}^\text{div}_{#1}} 
    \newcommand*{\bGdiv}[1][]{\bar{\mathcal{G}}^\text{div}_{#1}}

\DeclareMathOperator{\GL}{GL}
\DeclareMathOperator{\Dens}{Dens}
\DeclareMathOperator{\Sym}{Sym}
\DeclareMathOperator{\tr}{tr}
\DeclareMathOperator{\Div}{div}
\DeclareMathOperator{\cone}{cone}

\title{Simple unbalanced optimal transport}

\author{Boris Khesin, Klas Modin, and Luke Volk}
\address{B.K. and L.V.: Department of Mathematics, University of Toronto, ON M5S 2E4, Canada}
\email{khesin@math.toronto.edu {\rm and } luke.volk@mail.utoronto.ca}
\address{K.M.: Department of Mathematical Sciences, Chalmers University of Technology and University of Gothenburg, SE-412 96 Gothenburg, Sweden} \email{klas.modin@chalmers.se}

\date{}                                           

\begin{document}
\begin{abstract}
We introduce and study a simple model capturing the main features of unbalanced optimal transport. It is based on equipping the conical extension of the group of all diffeomorphisms with a natural metric, which allows a Riemannian submersion to the space of volume forms of arbitrary total mass. We describe its finite-dimensional version and present a concise comparison study of the geometry, Hamiltonian features, and geodesics for this and other extensions. One of the corollaries of this approach is that along any geodesic the total mass evolves with constant acceleration, as an object's height in a constant buoyancy field.
\end{abstract}

\maketitle

\section{Introduction}

Many problems of optimal transport are closely related to 
the differential geometry of diffeomorphism groups. In particular, the problem 
of moving one mass (or density) to another by a diffeomorphism while minimizing a certain (quadratic) cost can be understood as construction of geodesics in an appropriate metric on the space of normalized densities (or on its completion), see e.g.\ \cite{otto2001geometry, villani2003}. 
Similar problems arise in applications when one attempts to evaluate the proximity between different shapes or medical images \cite{Tr1998}. 
However, the action by a diffeomorphism does not allow a change of the total mass of the density. 
Hence, one arrives at the problem of constructing a natural extension of the action which would allow one to connect in the most economical way densities of different total masses. Such problems, first considered by Benamou~\cite{Be2003}, belong to the domain of \emph{unbalanced optimal transport} (UOT), and they have received considerable attention lately, see \cite{chizat2018unbalanced, gallouet2021regularity, kondratyev2016new, LiMiSa2018,piccoli2016properties, vialard2017diffeomorphisms} for geometry and analysis and \cite{BaHaKl2022,ChPeScVi2018,SeViPe2022} for numerical aspects.

Usually, the setting of unbalanced optimal transport involves a ``large'' extension $G=\Diff\ltimes C^\infty_+(M)$  of the group $\Diff$ of all diffeomorphisms of a manifold  by means of a semi-direct product with the space of smooth positive functions. Such a large semi-direct product group acts on densities by a change of coordinates and then by adjusting point-wise the obtained density by means of a function. 

In this paper, we instead introduce and study a much simpler ``small'' extension $\cone(\Diff)=\DiffR$ of the same group $\Diff$. 
This way, both the group of diffeomorphisms and the space of normalized densities have similar conical extensions $\cone(\Diff)$ and $\Vol=\cone(\Dens(M))$ by one extra parameter, the total mass $m$ of the density. We describe natural metrics and geodesics for those extensions. 

It turns out that the corresponding problem of unbalanced optimal transport, while being much easier to handle, captures most of the main features for the large extensions.  For instance, for both small and large extensions, a common phenomenon is that in many two-point problems a geodesic joining two end-densities goes through densities whose total mass dips below the smallest of the two it connects.
In particular, one of the corollaries of this approach is that along any geodesic the total mass $m$ evolves with constant acceleration, $\ddot m=const$, i.e., as an object's height  in a constant buoyancy field.

We also introduce special variables in which we demonstrate   the convexity of the dynamical formulation for the simple conical extension, generalizing 
 the convexity of standard optimal transport.    This convex minimization formulation is known to be central for the existence and uniqueness of the corresponding solutions in such variational problems.

One immediate additional advantage of the present approach is that it admits a finite-dimensional model,
where the diffeomorphism group $\Diff$ for $M=\mathbb R^n$ is replaced by its subgroup ${\rm GL}(n)$, while the space of all volume forms is constrained to its subspace of non-normalized Gaussian densities on $\mathbb R^n$.

Finally, we compare in more detail our ``small" extension with two other ``larger" extensions: the one considered 
in \cite{chizat2018unbalanced, gallouet2021regularity} and called Wasserstein-Fisher-Rao and the one which is indeed a weighted  sum of the Wasserstein and Fisher-Rao metrics. 
In a sense, those two models can be viewed as 
extensions of our simpler model in, respectively, Lagrangian and Hamiltonian settings, as we discuss below. 
We describe the corresponding geodesics and candidates for their finite-dimensional counterparts. It turns out that the corresponding larger finite-dimensional models are less natural than for the small extension, as they require additional restrictions on orbits of the corresponding action.

\smallskip

\textbf{Acknowledgements.} 
We would like to thank F.-X.~Vialard and T.~Gallou{\"e}t for thought-stimulating discussions.
The research of B.K.\ was partially supported by an NSERC Discovery Grant.
K.M.\ was supported by the Swedish Research Council (grant number 2022-03453) and the Knut and Alice Wallenberg Foundation (grant number WAF2019.0201).


\section{A conical extension of the diffeomorphism group}

Let $M$ be an $n$-dimensional Riemannian manifold with volume form $\mu$ of total volume (or ``mass'') equal to 1. 
Let $\Vol$ denote the set of all (un-normalized) volume forms on $M$ of finite total volume.
While for most applications one can think of a compact manifold $M$, it is also convenient to keep in mind the case of $M=\mathbb R^n$ with Gaussian densities on it.

Throughout the paper, we consider infinite-dimensional manifolds and Lie groups, such as the spaces of smooth normalized densities $\Dens(M)$, smooth volume forms $\Vol$, and smooth diffeomorphisms $\Diff$.
In the smooth, $C^\infty$ category, the manifold structures are modelled on Fréchet spaces.
Alternatively, one can work in the category of Banach manifold via completions in the Sobolev $H^s$ category (which requires $s>n/2+1$ as a consequence of the Sobolev embedding theorem).
For details on these settings, we refer to \cite{Ha1982,KhMiMo2020,Om2017} and references therein.

Let $\DiffR$ denote the direct product Lie group.
A left action of $\DiffR$ on $\Vol$ is given by $(\varphi,m)\cdot \varrho = m\, \varphi_*\varrho$ for 
$\varphi\in \Diff$, $\varrho\in \Vol$, and $m\in \mathbb R_+$. Fixing a Riemannian volume form $\mu$, this left action endows  the product $\Diff\times\mathbb R_+$ with the structure of a principal $G$-bundle with projection
    \begin{align*}
        \pi\colon \Diff\times\mathbb R_+ &\to \Vol \\ (\varphi,m) &\mapsto m\,\varphi_*\mu
    \end{align*}
and corresponding isotropy subgroup $G$ given by
\[
    G = \{ (\varphi, m)\mid m\, \varphi_*\mu = \mu \} = \SDiff\times \{ 1\}.
\]
It follows that $m=1$ by taking the integral.
The Lie algebra of $G$ is thus
\[
    \mathfrak{g} = \mathfrak{X}_\mu(M)\times \{ 0\}.
\]
The tangent space of the fibre through $(\varphi,m)$, denoted $\mathcal V_{(\varphi,m)} = \ker d\pi_{(\varphi,m)}$, gives the vertical distribution associated with the bundle $\pi$.

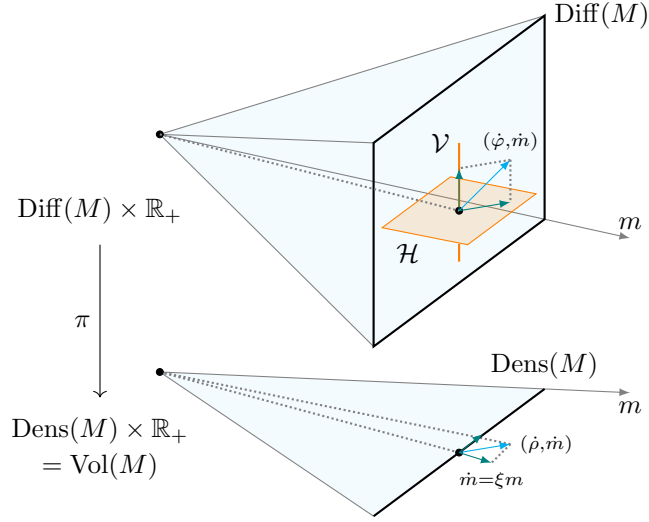
\begin{figure}
    \centering

    \begin{tikzpicture}[scale=2.25]
        \coordinate (A) at (-1.25,1.25); 
        \coordinate (B) at (-1.25,-0.15); 
        \coordinate (C) at (0.5,0.8); 
        \coordinate (D) at (0.5,-0.625); 

        \fill[cyan, fill opacity=0.05] (A) -- (0,0) -- (1,0.75) -- (1,1.95) -- cycle;
        \fill[cyan, fill opacity=0.05] (B) -- (0,-1) -- (1,-0.25) -- cycle;
        
        \draw[gray] (A) -- (0,0);
        \draw[gray] (A) -- (0,1.2);
        \draw[gray] (A) -- (1,1.95);
        \draw[gray, -{latex}] (A) -- (1.5,-2*1.5/9+35/36) node[above, black] {$m$};
        
        \filldraw (A) circle(0.02);
        \draw[thick] (0,0) -- (0,1.2) -- (1,1.95) node[right] {$\Diff$} -- (1,0.75)  -- cycle;

        \filldraw[orange, fill opacity=0.15] (0.1-0.05,0.7) -- (0.5-0.05,0.7 + 0.75*0.4) -- (0.9+0.05,0.9) -- (0.5+0.05,0.9 - 0.75*0.4) -- cycle;
        \node at (0.2,0.55) {$\mathcal H$};

        \draw[thick, orange] (C) -- (0.5,1.2);
        \draw[thick, orange] (0.5,0.605) -- (0.5,0.5);
        \node[left] at (0.5,1.2) {$\mathcal V$};

        \filldraw (C) circle(0.02);
        \draw[cyan,-{latex}] (C) -- ($(C)+(0.3,0.3)$) node[above, black] {$\scriptstyle(\dot\varphi,\dot m)$};
        \draw[teal,-{latex}] (C) -- ($(C)+(0.3,0.05)$);
        \draw[teal,-{latex}] (C) -- ($(C)+(0,0.25)$);

        \draw[thick, densely dotted, gray] (A) -- (C); 
        \draw[thick, densely dotted, gray] ($(D)+(0.3,0.05)$) -- (B);
        \draw[thick, densely dotted, gray] (B) -- (D); 
        \draw[thick, densely dotted, gray] ($(D)+(0.3,0.05)$) -- ($(D)+(0.19,-0.06)$); 
        \draw[thick, densely dotted, gray] ($(C)+(0.3,0.3)$) -- ($(C)+(0.3,0.05)$); 
        \draw[thick, densely dotted, gray] ($(C)+(0.3,0.3)$) -- ($(C)+(0,0.25)$); 

        \draw[gray] (B) -- (0,-1);
        \draw[gray, -{latex}] (B) -- (1.5,{-37/180-(2/45)*(1.5)}) node[black, below] {$m$};
        
        \filldraw (B) circle(0.02);
        \draw[thick] (0,-1) -- (1,-0.25) node[above] {$\Dens(M)$};
            
        \filldraw (D) circle(0.02);
        \draw[cyan, -{latex}] (D) -- ($(D)+(0.3,0.05)$) node[right, black] {$\scriptstyle(\dot\rho,\dot m)$};
        \draw[teal, -{latex}] (D) -- ($(D) + (0.139,0.11)$);
        \draw[teal, -{latex}] (D) -- ($(D)+(0.19,-0.06)$) node[below, black] {$\scriptstyle \dot m = \xi m$};

        \node at (-1.6,0.8) {$\Diff\times\mathbb R_+$};
        \draw[->] (-1.6,0.6) -- (-1.6,-0.3) node [midway, left] {$\pi$};
        \node at  (-1.6,-0.6) {$\displaystyle\genfrac{}{}{0pt}{}{\Dens(M)\times\mathbb R_+}{= \Vol}$};
    \end{tikzpicture}
    \caption{Illustration of the submersion between two conical extensions.}
    \label{fig:Conic-submersion}
\end{figure}

\begin{lemma}
    The vertical distribution for $\pi\colon\Diff\times\mathbb R_+\to\Vol$ is given by
        \[
            \mathcal V_{(\varphi,m)} =  \left\{(v\circ\varphi,0)\mid \operatorname{div}(\rho v) = 0 \quad\text{for}\quad \rho = \frac{m\,\varphi_*\mu}{\mu}\right\}.
        \]
\end{lemma}
\begin{proof}
    Given a curve $(\varphi(t),m(t)) \in \Diff\times\mathbb R_+$ with $(\varphi(0),m(0)) = (\varphi,m)$,  we have that for $(\dot\varphi(0),\dot m(0)) = (v\circ\varphi,\xi \,m )\in T_{(\varphi, m )}(\Diff\times\mathbb R_+)$:
    \begin{equation}\label{eq:bundle_projection_derivative}
        \begin{aligned}
            d\pi_{(\varphi, m )}(v\circ\varphi,\xi m ) &= \left.\frac{d}{dt}\right|_{t=0} m (t)\varphi(t)_*\mu \\
            &= \xi m \,\varphi_*\mu +  m \left.\frac{d}{dt}\right|_{t=0}\varphi(t)_*\mu \\
            &= \xi\rho\mu -  m\,  L_v\varphi_*\mu\,  = \, \xi\rho\mu -  m\, \Div(\rho v)\mu.
        \end{aligned}
    \end{equation}
    Thus, $(v\circ\varphi,\xi\, m )\in \mathcal V_{(\varphi, m )}$ if and only if $\xi\rho =  m \Div(\rho v)$. Now, by integrating the both sides against $\mu$ over $M$ we see that the integral of the divergence is zero. This implies that the constant $\xi=0$, which in turn implies that the divergence is zero point-wise. This concludes the proof.
\end{proof}

\subsection{A natural metric for unbalanced optimal transport (UOT)}

Consider the following metric on the direct product group $\DiffR$:
\begin{equation}\label{eq:new_metric_direct_product}
    \G[(\varphi, m )]((\dot\varphi,\dot m ), (\dot\varphi,\dot m )) =  m \int_M \lvert \dot\varphi\rvert^2 \mu + \frac{\dot m ^2}{ m }
    = \int_M \lvert v \rvert^2 \varrho +  m \xi^2
\end{equation}
for variables $v = \dot\varphi\circ\varphi^{-1}$, $\xi = \dot m / m $, and $\varrho =  m \varphi_*\mu$.

\begin{remark}\label{rem:change_to_cone}
Recall that for a Riemannian manifold $N$ with metric  $g(v,v)$ its conical extension $\cone(N):=N\times \mathbb R_+$ is a Riemannian manifold with metric $r^2g(v,v)+dr^2$.
Consequently, the above product group $\DiffR$ is a natural conical extension of the most straightforward $L^2$ metric on  $\Diff$ given by 
$$
 \langle \dot\varphi, \dot\varphi\rangle_{\varphi} =  \int_M \lvert \dot\varphi\rvert^2 \mu\,.
$$
Indeed, by changing variables $ m =r^2$ (implying $\dot  m =2r\dot r$) we come to the conical extension $\DiffR$ 
with metric
\begin{equation}\label{eq:new_metric_conic}
    \langle (\dot\varphi,\dot r), (\dot\varphi,\dot r)\rangle_{(\varphi,r)} = r^2\int_M \lvert \dot\varphi\rvert^2 \mu + 4{\dot r^2}\,.
\end{equation}
\end{remark}


The orthogonal complement of the vertical distribution with respect to the metric on $\Diff\times\mathbb R_+$ gives the horizontal distribution of the bundle.

\begin{lemma}\label{lem:hor_true_WFR}
    For the metric $\G$ in equation \eqref{eq:new_metric_direct_product}, the horizontal distribution at $(\varphi, m )$ is given by
    \begin{align*}
        \mathcal H_{(\varphi, m )} = \{ (\nabla \theta\circ\varphi, \xi m )\mid \theta\in C^\infty(M), \xi \in \mathbb{R}\} &\simeq \{ \theta\in C^\infty(M) \}\,, \\
     {where}\quad   \left(\nabla\theta\circ\varphi,\int_M\theta\varrho\right) & \leftrightarrow  \theta\,.
    \end{align*}
\end{lemma}

\begin{proof}
The vertical distribution $\mathcal V_{(\varphi, m )}$ consists of $(v\circ\varphi,0)$ where $v$ is divergence free with respect to $\varrho =  m \,\varphi_*\mu$, i.e., $\operatorname{div}(\rho v) = 0$.
    Thus it follows from the (generalized) Hodge decomposition and the choice of metric \eqref{eq:new_metric_direct_product}
   that if $(u\circ\varphi,\xi m ) \in \mathcal H$ then $u = \nabla \theta$ is a gradient vector field.
    It now follows that $(\nabla\theta\circ\varphi,\xi m )$ is orthogonal to $\mathcal V_{(\varphi, m )}$ for any $\xi\in\mathbb{R}$.
    In particular, we may encode $\xi$ in the arbitrary constant of $\theta$ for $\nabla\theta$.
    The choice $\xi m  = \int_M \theta\varrho$ gives a geometric identification of $\mathcal H_{(\varphi, m )}$ with the space $C^\infty(M)$.
\end{proof}
 
\begin{theorem}\label{thm:new_metric_induced}
    The metric $\G$ in \eqref{eq:new_metric_direct_product} projects as a Riemannian submersion to the metric $\bG$ on $\Vol$ given at any point
    $\rho\in \Vol$ by
    \begin{equation}\label{eq:new_metric_base}
        \bG[\varrho](\dot\varrho,\dot\varrho) = \int_M \left( \lvert \nabla\theta\rvert^2 + \xi^2 \right)\varrho , \quad \dot\rho = -\Div(\rho\nabla\theta) + \xi\rho ,\quad \int_M\dot\varrho =  m \xi . 
    \end{equation}
    Furthermore, the variable $\theta\in C^\infty(M)$, defined by the equations above together with 
    \begin{equation*}
        \xi  m  = \int_M \theta\varrho \, ,
    \end{equation*}  
    is Legendre-dual to $\dot\varrho$ under the pairing 
    \begin{equation*}
        \langle\dot\varrho,\theta \rangle = \int_M \theta \dot\varrho.
    \end{equation*}
    Consequently, the Hamiltonian on $T^*\Vol$ corresponding to the metric $\bG$ is
    \begin{equation}\label{eq:ham_new_metric}
        H(\varrho,\theta) = \frac{1}{2}\int_M \lvert \nabla\theta\rvert^2\varrho +  \underbrace{\frac{1}{2 m }\left( \int_M \theta\varrho \right)^2}_{\frac{1}{2} m \xi^2} .
    \end{equation}
\end{theorem}

\begin{proof}
    First, notice that the metric $\mathcal G$ is invariant under the right action of the isotropy subgroup $G$ on the tangent bundle $T(\operatorname{Diff}(M)\times \mathbb{R}_+)$.
    Thus, $\mathcal G$ is compatible with the principal bundle structure, so it indeed induces a metric $\bar {\mathcal{G}}$ on the base $\operatorname{Vol}(M)$. 
    Now take an arbitrary horizontal vector $(\nabla\theta\circ\varphi,\xi m ) \in \mathcal{H}_{(\varphi,m)}$.
    If $\varrho = \pi(\varphi,m)$ and $\dot\varrho = d\pi_{(\varphi,m)}(\nabla\theta\circ\varphi,\xi m )$ is the lifted bundle projection, then, by definition,
    \begin{equation*}
        \bar{\mathcal G}_\varrho(\dot\varrho, \dot\varrho) \equiv \mathcal{G}_{(\varphi,m)}(\nabla\theta\circ\varphi,\xi m ,\nabla\theta\circ\varphi,\xi m ) =
        \int_M \lvert \nabla\theta \rvert^2 \varrho +  m \xi^2.
    \end{equation*}
    From equation \eqref{eq:bundle_projection_derivative} for $d\pi$ we get that $\dot\rho = -\operatorname{div}(\rho\nabla\theta) + \xi\rho$.
    Applying integration and using that $m = \int_M \varrho$, we see that 
    \begin{equation*}
        \int_M \dot\varrho =  m \xi \, .
    \end{equation*}
    This confirms the formula~\eqref{eq:new_metric_base} for the induced metric.

    For the second statement, that $\theta$ is in fact the Legendre transform, the variable Legendre-dual to $\dot\varrho$ is defined by $\frac{\delta L}{\delta\dot\varrho}$ where $L$ is the Lagrangian corresponding to $\bar{\mathcal{G}}$. 
    Given a variation $\dot\varrho_\epsilon = \dot\varrho + \epsilon\,\delta\dot\varrho$ we obtain
    \begin{equation}\label{eq:Lagrangian}
        \left.\frac{d}{d\epsilon}\right|_{\epsilon=0}L(\varrho,\dot\varrho_\epsilon) = \underbrace{\int_M\left(\nabla\theta\cdot\nabla\left.\frac{d}{d\epsilon}\right|_{\epsilon=0}\theta_\epsilon\right)\,\varrho}_\text{(i)} + \underbrace{\xi m\left.\frac{d}{d\epsilon}\right|_{\epsilon=0}\xi_\epsilon}_\text{(ii)}.
    \end{equation}
    On the other hand, from the definition of $\theta$ in \eqref{eq:new_metric_base} we see that
    \begin{equation}\label{eq:variation_of_theta}
        \delta\dot\rho = \left.\frac{d}{d\epsilon}\right|_{\epsilon=0}\dot\rho_\epsilon = -\Div\left(\rho\nabla\left.\frac{d}{d\epsilon}\right|_{\epsilon=0}\theta_\epsilon\right) + \rho\left.\frac{d}{d\epsilon}\right|_{\epsilon=0}\xi_\epsilon . 
    \end{equation}
    By applying the divergence theorem to the term (i) and then comparing \eqref{eq:Lagrangian} with \eqref{eq:variation_of_theta}, we see that $\langle\delta\dot\varrho,\theta\rangle = \left.\frac{d}{d\epsilon}\right|_{\epsilon=0}L(\varrho,\dot\varrho_\epsilon)$, giving $\theta$ as the Legendre-dual variable of $\dot\varrho$.
    The form of the Hamiltonian follows readily.    
\end{proof}

Equipping $\Diff\times\mathbb R_+$ and $\Vol$ with the metrics $\G$  and $\bG$  (see \eqref{eq:new_metric_direct_product} and  \eqref{eq:new_metric_base}) makes $\pi\colon\Diff\times\mathbb R_+\to\Vol$ into a Riemannian submersion, which gives a correspondence between geodesics in $\Vol$ and  horizontal geodesics in $\Diff\times\mathbb R_+$ (i.e. those tangent to the horizontal distribution),  when given an initial point in the fiber.


\subsection{Dynamical and static formulations}
 
One can give the following dynamical formulation of conical unbalanced transport.

\begin{definition}
    The \emph{conical Wasserstein distance} $\mathit{WC}(\varrho_0,\varrho_1)$ between densities $\varrho_0,\varrho_1 \in \Vol$ (of possibly different total masses) is given by the following formula.
    \[
        \mathit{WC}^2(\varrho_0,\varrho_1) = \inf_{u,\xi, \varrho}\int_{0}^1 \left( \int_M \left(\lvert u\rvert^2 + \xi^2 \right)\varrho\right) \, dt\,,
    \]
    over time-dependent vector fields $u$, volume forms $\varrho$, and constants $\xi$ related by the constraints
    \[
        \dot\rho = -\Div(\rho u) + \xi\rho ,\quad \int_M\dot\varrho =  \xi \int_M \varrho , \quad \rho(0) = \rho_0,\quad \rho(1) = \rho_1.   
    \]  
\end{definition}

The convexity of the dynamical formulation of standard optimal transport, as studied by Benamou and Brenier~\cite{BeBr2000}, carries over to the conical extension.
Indeed, in the variables $\bar\rho = \rho/m = \frac{\varphi_*\mu}{\mu} > 0$, $w =\bar\rho \nabla\theta$, and $r = \sqrt{m} > 0$ it becomes
\begin{equation*}
    \mathit{WC}^2(\varrho_0,\varrho_1) = \inf_{w,\bar\rho, r} \int_{0}^1 \left( r^2 \int_M \frac{\lvert w\rvert^2}{\bar\rho}\mu + 4\dot r^2\right) \, dt\,,
\end{equation*}
i.e., a minimization of a convex functional,
under the affine constraints
\begin{align*}
    \dot{\bar\rho} + \operatorname{div} w = 0, \quad  &\bar\rho(0,\cdot) = \rho_0/m_0 ,\quad \bar\rho(1,\cdot) = \rho_1/m_1, \\ & r(0) = \sqrt{m_0}, \quad r(1) = \sqrt{m_1}. 
\end{align*}      
This convex minimization formulation is important for existence and uniqueness of solutions.


For the corresponding static formulation, the distance function $\mathit{WC}(\varrho_0,\varrho_1)$ is given in terms of the Riemannian metric \eqref{eq:new_metric_base} as
\[
    \mathit{WC}^2(\varrho_0,\varrho_1) = \inf_{\varrho}\int_{0}^1 
    \bG[\varrho](\dot\varrho,\dot\varrho) \; dt ,
\]
for curves $\varrho(t)$ with $\varrho(0) = \varrho_0$ and $\varrho(1) = \varrho_1$.
\smallskip

Note that the distance function $WC$ is necessarily implicit, as it depends on the metric on the Riemannian manifold $M$. It is bounded above via the (adjusted) Wasserstein distance $W$ between denities of the unit total mass
as follows: for densities $\varrho_0$ and $ \varrho_1$ of masses $m_0$ and $m_1$ respectively
one has the upper bound 
$$
WC^2(\varrho_0,\varrho_1)\le \min(m_0, m_1)\cdot W^2(\varrho_0/m_0,\varrho_1/m_1) +4(\sqrt{m_1}-\sqrt{m_0})^2\,.
$$
It follows from the orthogonality of the radial direction $r=\sqrt m$ to $\Dens(M)$.

\subsection{Geodesic equations}

The equations of geodesics for the above metrics can be computed in either Lagrangian or Hamiltonian form. The Lagrangian form of the geodesic equations, i.e. equations in the corresponding tangent bundle, for the conical manifold can be obtained using the formulas for warped Riemannian manifolds (see \cite{oneilsemi}). 
We review this approach in Appendix~\ref{appendix}. 
Here we derive the geodesic equations on the cotangent bundle, i.e., as the Hamiltonian equations for the Hamiltonian \eqref{eq:ham_new_metric}.

\begin{theorem}\label{thm:Hameqs_cone}
    The geodesic equations in the Hamiltonian form for the Hamiltonian \eqref{eq:ham_new_metric} are given by
    \begin{align*}
        \dot\rho &= -\operatorname{div}(\rho\nabla\theta) + \xi\rho \\
        \dot\theta &= -\frac{1}{2}\lvert \nabla \theta\rvert^2 - \xi\theta + \frac{\xi^2}{2}.
    \end{align*}
\end{theorem}

\begin{proof}
    Hamilton's equations are $\dot\varrho = \delta H/\delta \theta$ and $\dot\theta = -\delta H/\delta \varrho$. First, consider a variation $\theta_\epsilon = \theta + \epsilon\,\delta\theta$, where:
        \begin{align*}
            \left.\frac{d}{d\epsilon}\right|_{\epsilon=0}H(\varrho,\theta_\epsilon) 
            &= \int_M (\rho\nabla\theta)\cdot\nabla(\delta\theta)\,\mu + \frac{1}{m}\left(\int_M\theta\,\varrho\right)\left(\int_M\delta\theta\,\varrho\right) \\
            &= \int_M (\rho\nabla\theta)\cdot\nabla(\delta\theta)\,\mu + \int_M(\xi\rho)\delta\theta\,\mu \\
            &= \int_M \Big(\Div\left((\rho\nabla\theta)\delta\theta\right) - \Div(\rho\nabla\theta)\delta\theta + (\xi\rho)\delta\theta\Big)\,\mu \\
            &= \int_M \Big(- \Div(\rho\nabla\theta) + \xi\rho\Big)\delta\theta\,\mu,
        \end{align*}
    and so $\dot\rho = -\Div(\rho\nabla\theta) + \xi\rho$.

    Similarly, considering a variation $\varrho_\epsilon = \varrho + \epsilon\,\delta\varrho$:
        \begin{align*}
            \left.\frac{d}{d\epsilon}\right|_{\epsilon=0} H(\varrho_\epsilon,\theta)
            &= \frac12
            \int_M|\nabla\theta|^2\,\delta\varrho - \frac{\dot m_0}{2m^2}\left(\int\theta\,\varrho\right)^2 + \frac{1}{m}\left(\int_M\theta\,\varrho\right)\left(\int_M\theta\,\delta\varrho\right) \\
            &= \frac12
            \int_M|\nabla\theta|^2\,\delta\varrho - \frac{\dot m_0}{2m^2}(m\xi)^2 + \frac{1}{m}(m\xi)\left(\int_M\theta\,\delta\varrho\right),
        \intertext{but $\dot m_0 = \left.\frac{d}{d\epsilon}\right|_{\epsilon=0}\int_M\varrho_\epsilon = \int_M\delta\varrho$, so:}
            &= 
            \int_M\Big( \frac12|\nabla\theta|^2 - \frac{\xi^2}{2} + \xi\theta\Big)\,\delta\varrho,
        \end{align*}
    hence $\dot\theta = -\frac{1}{2}\lvert \nabla \theta\rvert^2 - \xi\theta + \frac{\xi^2}{2}$.
\end{proof}

Recall from above that $ m  = \int_M \varrho$ is the total volume and that $\xi = \int_M \theta\varrho /  m $ is the logarithmic derivative of $ m $.
The evolution of $ m $ and $\xi$ is described by the following theorem.

\begin{theorem}\label{thm:ham_geodesic}
    The variables $ m $ and $\xi$ satisfy the equations
    \begin{align*}
        \dot\xi &= \frac{1}{ m }\left( H(\varrho,\theta) -  m \xi^2 \right) \\
        \dot m  &=  m \xi .
    \end{align*}
    It can also be written as the second order equation
    \begin{equation*}
        \ddot m  = H .
    \end{equation*}
\end{theorem}
Since $H(\varrho,\theta)$ is constant along solutions, we obtain the following:
\begin{corollary}\label{cor:gravity}
The total volume $ m :=\int_M\varrho$ evolves with constant acceleration that  depends only on the energy level of the initial conditions.
\end{corollary}

In other words, the volume $m$ evolves as an object's height  in a constant  gravity or buoyancy field.
Note that in a  conical metric it is a common phenomenon that, depending of the boundary conditions, a geodesic joining two 
densities might enter the region where the total mass is smaller than the smallest of the two it connects.

\begin{proof}
    By construction $\dot m  = \xi m $.
    From $\xi m  = \int_{M} \theta\varrho$ we then get 
    \begin{align*}
        &\dot\xi = \frac{d}{dt} \frac{1}{ m } \int_M \theta\varrho = \frac{1}{ m } \int_M \left(  \dot\theta\varrho+\theta\dot\varrho\right) - \xi^2 \\
        &= \frac{1}{ m } \int_M \left(  \left( -\frac{1}{2}\lvert \nabla \theta\rvert^2 - \xi\theta + \frac{\xi^2}{2} \right)\varrho + \theta\left( -\operatorname{div}(\rho\nabla\theta) + \xi\varrho \right)\right) - \xi^2 \\
        &= \frac{1}{ m } \int_M \left( \frac{1}{2}\lvert \nabla \theta\rvert^2  \right)\varrho - \frac{\xi^2}{2}
        =  \frac{1}{ m } \left( \int_M  \frac{1}{2}\lvert \nabla \theta\rvert^2  \varrho + \frac{1}{2} m \xi^2 - \frac{1}{2} m \xi^2  \right) - \frac{\xi^2}{2} \\
        &= \frac{H(\varrho,\theta)}{ m } - \xi^2 .
    \end{align*}
    We then obtain that
    $
        \ddot m  = \dot m  \xi +  m \dot\xi =  m \xi^2 +  m  \frac{1}{ m }(H -  m \xi^2) = H.
    $
\end{proof}

\begin{remark}
    Note from equation \eqref{eq:ham_new_metric} that $H \geq \frac{1}{2}m\xi^2$ with equality if and only if $\theta$ is constant, which corresponds to the invariant subset of pure scalings of the density $\varrho$.
\end{remark}

\begin{remark}\label{rem:production}
    In the simple model just presented, mass is added or removed proportionally to $\rho$.
    It is easy to modify the model, so it has a localized ``production function'' $f = f(x) \geq 0$
    representing a fixed rate of supply or demand distribution over $M$.
    In the model above such a rate was constant, $f\equiv 1$, manifesting that the volume was added or subtracted uniformly over $M$, while using a nonconstant $f$ one can adjust the UOT model and make some regions of $M$ prefered to others.  
    Then, instead of the Hamiltonian \eqref{eq:ham_new_metric} we use 
    \begin{equation}\label{eq:ham_new_metric_reserve}
        H(\varrho,\theta) = \frac{1}{2}\int_M \lvert \nabla\theta\rvert^2\varrho +  \frac{1}{2 m }\left( \int_M \theta f \varrho \right)^2 ,
    \end{equation}
    where now $\xi m = \int_M \theta f \varrho$.
    The evolution of $\rho$ and $\theta$ then becomes
    \[
        \dot\rho = -\operatorname{div}(\rho\nabla\theta) + \xi f\rho  \, ,\qquad
        \dot\theta = - \frac{1}{2}\lvert \nabla\theta \rvert^2 - \xi f \theta + \frac{\xi^2}{2}.
    \]
\end{remark}

\begin{remark}
    Note that the geodesic equation in Theorem~\ref{thm:Hameqs_cone} retains a property of conical extensions; the radial projection of a geodesic curve in $\Vol$ corresponds to a Wasserstein geodesic on the space $\Dens(M)$ of normalized densities, albeit in a different parameterization and with a different total length.
    Indeed, it follows from the fact that a totally geodesic submanifold of a manifold remains totally geodesic after its conical extension, see the next lemma.
    Note, however, that this projection property for geodesics does not hold for extensions with non-constant production functions, cf.\ Remark~\ref{rem:production}.
\end{remark}

\begin{remark}
Another model for unbalanced optimal transport is given in \cite{Gangbo}.
It is also an extension by means of one extra dimension, and it can be viewed as a {\it cylindrical}-type, rather than {\it conical},
extension of the Wasserstein geometry discussed above. The dynamics of density is given by the equation
$ \dot\rho = -\Div(\rho u) + h$ where the last term $h=h(t)$ can be regarded as ``pumping” the constant density 
over the whole of manifold, and it replaces the linear term $\xi\rho$ proportional to the current density $\rho$ in the conical model. Then the dynamics of this extra variable $h(t)$ is governed by the vector field $u$ fulfilling the inviscid Burgers  equation, as in the standard optimal transport. This leads to the uniform change of $h(t)$ ($\dot h={\rm const}$) and
a somewhat peculiar numerical behavior observed in \cite{Gangbo}.
\end{remark}


\subsection{A finite-dimensional version of the simple  UOT}\label{sect:finite1D}

The existence of a finite-dimensional version of the conical extension is based on the following observation.

\begin{lemma}\label{lem:total-geodesic}
Suppose that a submanifold $N\subset M$ is a totally geodesic in the manifold  $M$. Then $\cone(N)$ is totally geodesic
in $\cone(M)$.
\end{lemma}

\begin{proof}
To prove the totally geodesic property, one needs to compute the geodesic equations. One can see that if the covariant derivatives
$\nabla_{\dot q} \dot q$ for $q\in N\subset M$ belong to the tangent bundle of $N$ then its extension by the radial variable $r\in \mathbb R_+$ can belong to the 
product of the tangent bundle of $N$ and $\mathbb R_+$. 
\end{proof}

\begin{corollary}
Conical extensions $\GL(n)\times \mathbb R_+\subset \DiffRn\times \mathbb R_+$ of the sub-manifolds $\GL(n)\subset \DiffRn$ are totally geodesic for the natural UOT metric.\footnote{The same statement holds for the unbalanced $\dot H^1$ and Fisher-Rao metrics considered below.}
\end{corollary}

On the base, we now restrict the metric to the space of scaled (or non-normalized)  Gaussian densities $\mathcal N\subset \Vol$.
In the total space, we restrict the metric to the finite-dimensional direct product subgroup $\GL(n)\times \mathbb R_+\subset \DiffR$:
     \begin{equation}\label{eq:GLn_metric1}
        \Gfin[(A, m )]((\dot A,\dot m ),(\dot A,\dot m )) =  m \int_{\mathbb R^n}\|\dot Ax\|^2\eta(x) + \frac{\dot m ^2}{ m },
     \end{equation}
where $\varphi(x) = Ax$ for $A\in\GL(n)$ and $\eta = p(x,\Sigma)\,dx$ is a normal density with covariance matrix $\Sigma$ and zero mean,
    \[p(x,\Sigma) = \frac{1}{\sqrt{(2\pi)^n|\Sigma|}}\exp\left(-\frac12 x^\top\Sigma^{-1} x\right).\]

\begin{remark}\label{rem:unscaled_gaus}
Recall that the isotropic Gaussian is given by
    \[
        \mu(x) = \frac{1}{\sqrt{(2\pi)^n}} \exp\left(-\frac{1}{2}x^\top x \right)\, dx.
    \]
Consider now a group element $(\varphi: x\mapsto Ax, m )$. The action on $\mu$ is 
    \[
         m \, \varphi_*\mu = \sqrt{\frac{ m ^2}{\det(AA^\top)(2\pi)^n}}\exp\left(-\frac{1}{2}x^\top(AA^\top)^{-1}x \right)\, dx
        =:  m \, p(x,\underbrace{AA^\top}_{\Sigma})dx.
    \]
The latter has the natural scaling property:
\[
         m \, p(x,\Sigma) = p\left( \frac{x}{\sqrt[n]{ m }}, \frac{\Sigma}{\sqrt[n]{ m ^2}}\right).
    \]
    Note that one cannot write $ m \, p(x,\Sigma) = p(x,\tilde\Sigma)$ for some covariance matrix $\tilde\Sigma$, since $p(\cdot,\Sigma_1) = p(\cdot,\Sigma_2) \iff \Sigma_1 = \Sigma_2$.
\end{remark}

    After identifying the Gaussian densities with their (symmetric positive definite) covariance matrices in $\Sym_+(n)$, the finite-dimensional version of our bundle is 
        \begin{align*}
            \pi\colon \GL(n)\times\mathbb R_+ &\to \Sym_+(n)\times\mathbb R_+ \\
            (A, m ) &\mapsto (A\Sigma A^\top, m )\,,
        \end{align*}
    where we parametrize the base using both the covariance matrix and the total volume. 
    The metric \eqref{eq:GLn_metric1} on $\mathrm{GL}(n)\times \mathbb{R}_+$ in terms of $(A,V)\in T_A\GL(n)$, where $V= \dot A A^{-1}$, and $( m ,\xi)\in T_ m \mathbb R_+$ is given by
    \begin{align*}
        \Gfin[(A, m )]((\dot A,\dot m ),(\dot A,\dot m )) &= 
         m  \left( \int_{\mathbb{R}^n} |V x|^2 p(x, \Sigma)dx + \xi^2\right) \\ &=  m  \left(\tr(\Sigma V^\top V) + \xi^2 \right).
    \end{align*}

    \begin{lemma}
        The vertical and horizontal distributions of $\GL(n)\times\mathbb R_+$ with the metric $\Gfin$ are given by:
            \begin{align*}
                \mathcal V_{(A, m )} &= \{(VA,0) \in T_A\GL(n)\times\mathbb R \mid 0 = V(A\Sigma A^\top) + (A\Sigma A^\top)V^\top\}, \\[0.5em]
                \mathcal H_{(A, m )} &= \{(VA,\xi m ) \in T_A\GL(n)\times\mathbb R \mid V\in\Sym(n),\quad\xi\in\mathbb R\}.
            \end{align*}
    \end{lemma}
    \begin{proof}
        If $(A(t), m (t))$ is a path in $\GL(n)\times\mathbb R_+$ with $(A(0), m (0)) = (A, m )$ and $(\dot A(0),\dot m (0))=(VA,\xi m )$, then:
            \begin{align*}
                d\pi_{(a, m )}(VA,\xi m )
                &= \left.\frac{d}{dt}\right|_{t=0}(A(t)\Sigma A(t)^\top, m (t)) \\
                &= (\dot A(0)\Sigma A(0)^\top + A(0)\Sigma\dot A(0)^\top,\dot m (0)) \\
                &= (VA\Sigma A^\top + A\Sigma A^\top V^\top,\xi m ),
            \end{align*}
        which gives the desired vertical distribution as its kernel. Noting that $\mathcal V_{(A, m )}$ consists of $VA$ such that $VA\Sigma A^\top$ is antisymmetric, if $WA\in\mathcal H_{(A, m )}$ then for all such $Z = VA$ we have:
            \[0 = \Gfin[(A, m )]((W,\xi m ),(Z,0)) =  m \tr(WA\Sigma Z^\top) = - m \tr\left(W(Z\Sigma A^\top)\right).\]
        Picking $Z\Sigma A^\top$ to be the elementary antisymmetric matrix with $1$ in the $(i,j)$-entry and $-1$ in the $(j,i)$-entry (for $i\neq j$) gives that $W$ must be symmetric, giving the desired horizontal distribution.
    \end{proof}

    The projection $\pi\colon\GL(n)\times\mathbb R_+\to \Sym_+(n)\times\mathbb R_+$ subduces a metric $\bGfin$ on $\Sym_+(n)\times\mathbb R_+$ by defining
        \[\bGfin[\pi(A, m )](d\pi_{(A, m )}(X,a),d\pi_{(A, m )}(Y,b)) = \Gfin[(A, m )]((X,a)_{\mathcal H},(Y,b)_{\mathcal H}),\]
    where the subscript $\cdot_\mathcal{H}$ denotes the horizontal part of the vector. 
    This metric makes $\pi$ into a Riemannian submersion. 
    Explicitly, 
        \[\bGfin[(V, m )]((X,\xi m ),(X,\xi m )) =  m (\tr(VSS) + \xi^2),\]
    where $S$ is a symmetric $n\times n$ matrix that is a solution to the continuous Lyapunov equation given by $X = SV + VS$. 
    The finite-dimensional metric so constructed is simply the cone metric built from the ``balanced'' case described in \cite{modin2016geometry}.


    \bigskip
    
    Let us now compute the Legendre transform.
    The dual variable $P$ to $\dot V = X$ is given by
        \begin{align*}
            \langle P,\delta\dot V \rangle &= \frac{d}{d\epsilon} \frac{1}{2}\bGfin[(V,m)]((\dot V_\epsilon,\dot m),(\dot V_\epsilon,\dot m))
        = \frac{m}{2} \tr(\Sigma (\delta S\,S + S\,\delta S))\\
        &= \frac{m}{2} \tr((V \delta S  + \delta SV) S) = \frac{m}{2} \tr(S \delta\dot V),
        \end{align*}
    where for $\delta S$ we have
    \[
        \delta\dot V = \delta SV + V\,\delta S .
    \]
    Thus, the dual variable is $P = mS/2$.
    The dual variable for $m$ is $\xi$.
    This gives the Hamiltonian
    \[
        H(V, m ,P,\xi) = \frac{\tr(VPP)}{2m} + \frac{1}{2} m\xi^2 .
    \]
    This gives the Hamiltonian form of the geodesic equations on $T^*(\Sym_+(n)\times \mathbb{R}_+)$ as
    \begin{align*}
        &\dot V = \frac{2}{ m }\left(PV + VP\right), & &\dot m  = \xi m  \\
        &\dot P = -\frac{P^2}{2 m }, & &\dot\xi = \frac{1}{2}\left(\frac{\operatorname{tr}(vP^2)}{ m ^2} - \xi^2 \right)
    \end{align*}

    \begin{remark}
        Note that here the need for all four equations, as opposed to Theorem \ref{thm:ham_geodesic} where we only have two equations,  arises from the observation in Remark \ref{rem:unscaled_gaus} that we need two parameters to describe the unscaled Gaussian distributions.
    \end{remark}

            
\subsection{Affine transformations and Gaussians with nonzero means}

It turns out that considering the group of affine transformations $\GL(n)\ltimes \mathbb R^n\subset  \DiffRn $  acting on Gaussians with arbitrary (not necessarily zero) means does not essentially change the above picture. While the group extension is semi-direct, its metric extension is a direct product, provided that a reference Gaussian is $\eta = p(x,\Sigma)\,dx$ with mean $\mu = 0$.

\begin{remark}
    For a more general reference Gaussian measure the metric   accumulates the following terms:
    \begin{multline*}
        \Gaff[(A,b, m )]((\dot A,\dot b,\dot m ),(\dot A,\dot b,\dot m )) = \\  m \left[\tr(\dot A\Sigma\dot A^\top) + \|\dot A\mu\|^2 + 2\langle\dot A\dot b,\mu\rangle + \|\dot b\|^2\right] + \frac{\dot m ^2}{ m },
    \end{multline*}
    and it descends to the metric $\bGaff$ on $(\GL(n)\ltimes\mathbb R^n)\times\mathbb R_+$ given by:
    \begin{multline*}
	    \bGaff[(U,v, m )]((X,y,a),(W,z,b))= \\
			  m \left[\tr(XU^{1/2}\Sigma U^{1/2}W^\top) + \langle XU^{1/2}\Sigma^{1/2}\mu,WU^{1/2}\Sigma^{1/2}\mu\rangle\right. \\
		\left.+ \langle XU^{1/2}\Sigma^{1/2}z,\mu\rangle + \langle WU^{1/2}\Sigma^{1/2}y,\mu\rangle + \langle y,z\rangle\right] + \frac{ab}{ m }.
    \end{multline*}
    Note that if $\mu = 0$, then several terms vanish, and one is left with the product metric of $(\Sym_+(n)\times\mathbb R^n)\times\mathbb R_+$.    
\end{remark}

This implies that the geodesics between two Gaussian densities with different means are the pushforwards of measures by affine transformations, which decompose into the uniform motion between the centers of the two Gaussian densities and the  $\GL(n)$ transformation  with the fixed center. 

\medskip

\begin{remark}
    The explicit geodesics for $\Sym_+(n)$ with the Wasserstein metric are given in McCann \cite{mccann1997convexity}. 
    In particular, for $U,V\in \Sym_+(n)$, define
        \[T = U^{1/2}(U^{1/2}VU^{1/2})^{-1/2}U^{1/2}\in \Sym_+(n),\]
    and then $W(t) = [(1-t)E + tT]V[(1-t)E + tT]$ is a geodesic between $U$ and~$V$. In our case, if the reference measure is of mean zero ($\mu=0$) the geodesics in the \emph{balanced} affine extension are those of the product $\Sym_+(n)\times\mathbb R^n$. The geodesics in the unbalanced case are those of the conical extension  $\Sym_+(n)\times\mathbb R^n\times \mathbb R_+$.

    The sectional curvatures of $\Sym_+(n)$ with the Wasserstein metric are well understood (see \cite{asuka}) and are known to be non-negative. 
    Hence, in the case  $\mu = 0$ the affine and conical extensions also have non-negative sectional curvatures.
\end{remark}



\section{A ``large" extension for UOT}

\subsection{The ``large" group, metric, and the geodesic equations}
A more ``classical'' approach to unbalanced optimal transport involves the following large semi-direct extension of the group $\Diff$ of all diffeomorphisms of a manifold by means of the space of smooth functions, see, e.g.,
 \cite[\S 3.2.2]{vialard2017diffeomorphisms}.
Namely, the semi-direct product $G=\Diff\ltimes C^\infty_+(M)$ acts on $\Vol$ by $(\varphi, \lambda )\cdot \varrho = \varphi_*(\lambda\varrho)$, i.e.,  diffeomorphisms  act on densities by changes of coordinates, while functions adjust the obtained density point-wise. 
Let $\mu\in\Vol$ denote the reference volume form.
Then we get a projection $\Pi\colon G\to \Vol$ by the action on $\mu$.

\begin{lemma}
    The vertical bundle is given by
    \begin{equation*}
        \mathcal V_{(\varphi,\lambda)} = \{ (v\circ\varphi,\frac{\varphi^*(L_v \varrho)}{\mu}) \mid v\in \mathfrak{X}(M)\} \simeq \mathfrak{X}(M),
    \end{equation*}
    where $\varrho = \varphi_*(\lambda\mu)$.
\end{lemma}

\begin{proof}
    A curve $(\varphi(t),\lambda(t))$ belongs to the fiber of $\varrho\in\Vol$ iff  $\varphi(t)_*(\lambda(t)\mu) = \varrho$ for all $t$.
    Equivalently,
    \[
        \lambda(t) = \frac{\varphi(t)^*\varrho}{\mu}.
    \]
By differentiating this relation  we get the result.
\end{proof}

\begin{remark}
    This description of $\mathcal V$ is equivalent to the one given by Vialard~\cite{vialard2017diffeomorphisms} as
    \[\ker d\pi(\varphi,\sqrt{\text{Jac}\,\varphi}) = \left\lbrace\left.\left(v,\frac{\Div v}{2}\right)\circ(\varphi,\sqrt{\text{Jac}\,\varphi})\,\right|\, v\in\mathfrak X(M)\right\rbrace\]
    The relation is $\varrho = \operatorname{Jac}(\varphi) \mu$, where the square root appears if one passes from volume forms to \emph{half-densities}, i.e., geometric objects that transform as the square root of a volume form.
\end{remark}

Consider now the Riemannian metric on $G$ considered in \cite{chizat2018unbalanced, kondratyev2016new, piccoli2016properties, vialard2017diffeomorphisms}
and given by
\begin{equation}\label{eq:semi_direct_metric}
    \Glarge[(\varphi,\lambda)]((\dot\varphi,\dot\lambda), (\dot\varphi,\dot\lambda)) = \int_M \lvert \dot\varphi\rvert^2 \lambda \mu + \int_M \frac{\dot\lambda^2}{\lambda}\mu .
\end{equation}

\begin{lemma}
    The horizontal bundle of the metric $\Glarge$~\eqref{eq:semi_direct_metric} is
    \begin{equation*}
        \mathcal H_{(\varphi,\lambda)} = \{ (\nabla \theta\circ\varphi,\lambda (\theta\circ\varphi))\mid \theta\in C^\infty(M) \} \simeq C^\infty(M).
    \end{equation*}
\end{lemma}

\begin{proof}
    Any element in $T_{(\varphi,\lambda)} G$ can be written $(u\circ \varphi, \lambda (\theta\circ\varphi))$.
    Suppose that $(u\circ\varphi, \lambda(\theta\circ\varphi))\in \mathcal H_{(\varphi,\lambda)}$. 
    Since for all $v\in\mathfrak X(M)$ the pairs $(v\circ\varphi, \varphi^*(L_v\varrho)/\mu)$ span the vertical space $\mathcal V_{(\varphi,\lambda)}$,
    we have that
        \begin{align*}
            0
            &= \Glarge[(\varphi,\lambda)]\left(\left(v\circ\varphi,\frac{\varphi^*(L_v\varrho)}{\mu}\right),(u\circ\varphi,\lambda(\theta\circ\varphi))\right) \\
            &= \int_M\langle v\circ\varphi,u\circ\varphi\rangle\, \lambda\mu + \int_M(\theta\circ\varphi)\varphi^*(L_v\varrho)\\
            &= \int_M\langle v,u\rangle\, \varphi_* (\lambda\mu) + \int_M\theta L_v\varrho \\
            &= \int_M\langle v,u\rangle\, \varrho - \int_M\langle v, \nabla \theta\rangle \varrho \, =\, \int_M\langle v,u - \nabla \theta \rangle\, \varrho  .
        \end{align*}
The latter integral vanishes for any $v\in \mathfrak{X}(M)$ if and only if $u = \nabla\theta$, which concludes the proof.
\end{proof}

\begin{theorem}[cf.~\cite{vialard2017diffeomorphisms}]\label{thm:uot_geom}
    The metric $\Glarge$ given by \eqref{eq:semi_direct_metric}
    projects as a Riemannian submersion to the metric $\bGlarge$ on $\Vol$ given by
    \begin{equation*}
        \bGlarge[\varrho](\dot\varrho,\dot\varrho) = \frac{1}{2}\int_M \left( \lvert \nabla\theta\rvert^2 + \theta^2 \right)\varrho , \qquad \dot\rho = -\operatorname{div}(\rho\nabla\theta) + \rho\theta.
    \end{equation*}
    The variable $\theta\in C^\infty(M)$ is Legendre-dual to $\dot\varrho$, i.e., the Hamiltonian corresponding to the metric is
    \begin{equation}\label{eq:WFR-hamilt}
        H(\varrho,\theta) = \frac{1}{2}\int_M \left( \lvert \nabla\theta\rvert^2 + \theta^2 \right)\varrho .
    \end{equation}
    The equations of geodesics (in Hamiltonian form) are
    \begin{align*}
        &\dot\rho = -\operatorname{div}(\rho\nabla\theta) + \rho\theta \\
        &\dot\theta = -\frac{1}{2} \lvert \nabla\theta\rvert^2 - \frac{\theta^2}{2} .
    \end{align*}
\end{theorem}

\begin{proof}
    The proof of this follows similarly to Theorem \ref{thm:Hameqs_cone}. Hamilton's equations are $\dot\varrho = \delta H/\delta \theta$ and $\dot\theta = -\delta H/\delta \varrho$. Given a variation $\theta_\epsilon = \theta + \epsilon\,\delta\theta$, note:
        \begin{align*}
            \left.\frac{d}{d\epsilon}\right|_{\epsilon=0} H(\varrho,\theta_\epsilon)
            &= \int_M \rho(\nabla\theta\cdot\nabla(\delta\theta) + \theta\,\delta\theta)\,\mu \\
            &= \int_M \Big(-\Div(\rho\nabla\theta) + \rho\theta\Big)\delta\theta\,\mu,
        \end{align*}
    and so $\dot\rho = -\Div(\rho\nabla\theta) + \rho\theta$.

    Similarly, considering a variation $\varrho_\epsilon = \varrho + \epsilon\,\delta\varrho$, we see:
        \[
            \left.\frac{d}{d\epsilon}\right|_{\epsilon=0} H(\varrho_\epsilon,\theta)
            = \int_M\frac12\Big(|\nabla\theta|^2 + \theta^2\Big)\,\delta\varrho,
        \]
    and so we immediately get $\dot\theta = - |\nabla\theta|^2/2 - {\theta^2}/{2}$.
 \end{proof}

\begin{remark} 
The metric $\bGlarge$ in Theorem~\ref{thm:uot_geom} can be interpreted as an interpolation between Wasserstein--Otto and Fisher--Rao, but not a convex combination of the Riemannian metric tensors (see Remark~\ref{rem:lagrange_vs_hamiltonian} below).
One way of understanding the relation is the following: the Wasserstein--Otto part of the metric depends on the finite-dimensional metric $g$ on $M$, but the second term does not. Thus, let us introduce a parameter $\beta$ by making the replacement $g\mapsto \beta g$. 
Then, as $\beta\to \infty$, we recover the Fisher--Rao metric (indeed, in the Hamiltonian \eqref{eq:WFR-hamilt} the term with $\nabla \theta\to 0$ for the metric $\beta g$ and only the second term remains, which corresponds to the Fisher--Rao metric).
On the other hand, as $\beta\to 0$ we recover the (scaled) Wasserstein--Otto metric (represented by the first term in the Hamiltonian).
Thus, this mixed metric behaves as Fisher--Rao on small scales, but as Wasserstein--Otto on large scales.
\end{remark}

\begin{remark}
    The metric in Theorem~\ref{thm:uot_geom} lifted to a  metric on $\Diff\times\mathbb{R}_+$ is given by
    \begin{equation*}
        \langle (v\circ\varphi,\dot\lambda), (v\circ\varphi,\dot\lambda)\rangle_{(\varphi,\lambda)} =  \int_M \left( \lvert \nabla\theta\rvert^2 + \theta^2 \right)\varrho ,
    \end{equation*}
    where $\varrho = \lambda\varphi_*\mu$ and $\theta \in C^\infty(M)$ is the solution to the equation
    \begin{equation*}
        -\operatorname{div}(\rho\nabla\theta) + \rho\theta = -\operatorname{div}(\rho v) + \frac{\rho\dot\lambda}{\lambda} .
    \end{equation*}
    Notice, however, that $\theta$ is somewhat difficult to find, as it requires the solution of a non-local equation.
    The ``small'' extension discussed above does not encounter this difficulty. 
This shows that {\it the metric of the simple UOT is not a restriction} of the noticeably more complicated  metric in Theorem~\ref{thm:uot_geom}.
\end{remark}

\subsection{Interrelation between the small and large extensions}
We discussed above, in Remark \ref{rem:production}, that in the simple UOT one can introduce a localized ``production function'' $f \geq 0$, which leads to the evolution of density $\rho$ given by 
    \[
        \dot\rho = -\operatorname{div}(\rho\nabla\theta) + \xi f\rho  .
    \]
    Furthermore, one can consider a model with several production functions $f_1, \dots, f_k$, each with an independent coefficient, which are optimised together. This way one can view this adjusted UOT as an approximation of the unbalanced transport corresponding to the large extension and the metric  $\bGlarge$ on $\Vol$, cf. Theorem \ref{thm:uot_geom} and see \cite{vialard2017diffeomorphisms}.

    Indeed, for several production functions the  Hamiltonian \eqref{eq:ham_new_metric_reserve} becomes 
    \[ 
        H(\varrho,\theta) = \frac{1}{2}\int_M \lvert \nabla\theta\rvert^2\varrho +  \sum_{i=1}^k\frac{1}{2 m }\left( \int_M \theta f_i \varrho \right)^2 ,
    \] 
    where $m$ is the total mass and $\xi_i m = \int_M \theta f_i \varrho$. Then the 
    evolution of $\rho$ is given by 
    \[
        \dot\rho = -\operatorname{div}(\rho\nabla\theta) + \rho \sum_{i=1}^k \xi_i f_i \,
    \]
where the last term $\rho\sum_1^k \xi_i f_i$ can be regarded as a finite-dimensional replacement (``approximation") of the term $\rho\theta$ with a function $\theta\in C^\infty(M)$ in the evolution of density 
$$\dot\rho = -\operatorname{div}(\rho\nabla\theta) + \rho\theta $$ given in Theorem \ref{thm:uot_geom}.
Thus, as the number $k$ of fixed production functions goes to infinity, one recovers the problem of optimal transport with variable production of density over $M$. Such an approximation can be useful for numerical modeling.





\subsection{A finite-dimensional version of the large extension}

Consider the finite-dimensional group
$\GL(n)\ltimes \Sym(n)\subset  \DiffRn \ltimes C^\infty_+(\mathbb R^n)$,
where $\Sym(n)$ is the additive space of symmetric $n\times n$ matrices (or equivalently, the corresponding quadratic forms on $\mathbb R^n$), on which linear transformations act by the variable change. 
\medskip

 We regard $\Sym(n)$ as a subset of $C^\infty_+(\mathbb R^n)$ by using the map
    \[E = \{x\mapsto\exp(x^\top Sx) \mid S\in\Sym(n)\}\subset C^\infty_+(\mathbb R^n)\,.\]
This way the addition group of symmetric matrices $\Sym(n)$ becomes a multiplication subgroup of
positive functions $C^\infty_+(\mathbb R^n)$.
An advantage of this approach is that the Riemannian submersion can be restricted to the finite-dimensional
model, where the group $\GL(n)\ltimes E$ acts on $E\subset \Vol$. Here an element $(A,S)\in \GL(n)\ltimes E$
acts naturally on a density
$p(x,\Sigma)$ by 
changing variables and bringing the quadratic into the exponential:
      \begin{align*}
      (A, S): ~~ p(x,\Sigma)\mapsto \, \widetilde p(x,\Sigma):=&\exp(x^\top S x)\, p(x,A^\top\Sigma A)\\
     =& \frac{1}{\sqrt{(2\pi)^n|\Sigma|}}\exp \left(-\frac12 x^\top((A^\top\Sigma A)^{-1}  - 2S)x\right).
    \end{align*}
The drawback is that even if $p(x,\Sigma)$ is a Gaussian density and $S$ is also positive-definite,
 the symmetric matrix $(A^\top\Sigma A)^{-1} - 2S$ might not be positive-definite! 
 This means that the total volume of the density
$\widetilde p(x, \Sigma )=\break \exp(x^\top S x)\, p(x,A^\top\Sigma A) $ in $\mathbb R^n$  might be infinite.
Thus one has to consider a restricted orbit of this action, constrained by the condition of positivity of the
matrix $(A^\top\Sigma A)^{-1}  - 2S$. 

Note that the required positivity is automatically satisfied and does not constrain anything 
in the infinite-dimensional setting of the space of $L^2$ densities $\Vol$. Also this constraint is not required in the finite-dimensional simple 1D conical extension described in Section \ref{sect:finite1D}.


\section{Other versions of the unbalanced transport metric}

\subsection{A ``small'' extension with a divergence term}
Consider now the Riemannian metric on $\DiffR$ given by 
\begin{equation}\label{eq:trueWFR}
    \Gdiv[(\varphi,\lambda)]\left( (v\circ\varphi,\xi\lambda), (v\circ\varphi,\xi\lambda) \right)
    = \frac{1}{2}\int_M \left( \lvert v\rvert^2 + \frac{\operatorname{div}(\rho v)^2}{\rho^2} \right)\varrho + \lambda\xi^2.
\end{equation}

Notice the similarity of  $\Gdiv$ with $\G$ given by \eqref{eq:new_metric_direct_product}: it is the same metric supplemented by the divergence term. In particular, on vertical vectors it is exactly the same metric.
Thus, the horizontal bundle is the same as in Lemma~\ref{lem:hor_true_WFR}.

Consider now the metric $\bGdiv$ on $\Vol$ given by
\[
    \bGdiv[\varrho](\dot\varrho,\dot\varrho) = \int_M \left(  \lvert \nabla S\rvert \varrho + \left(\frac{\dot\varrho}{\varrho}\right)^2\varrho \right) ,\quad -\operatorname{div}(\rho\nabla S) = \dot\rho - \kappa\rho \; .
\]
Here, we think of $\kappa$ as a Lagrange multiplier to ensure that the average of the right-hand side vanishes.

\begin{theorem}
    The projection $\pi\colon \DiffR \to \Vol$ given by $\pi(\varphi,\lambda) = \lambda\varphi_*\mu$ is a Riemannian submersion with respect to $\Gdiv$ and $\bGdiv$.
\end{theorem}

\begin{proof}
    The tangent derivative of the projection is 
    \[
        T_{(\varphi,\lambda)}(v\circ\varphi,\xi\lambda) = \xi\varrho -  L_v\varrho  .
    \]
    In particular, for a horizontal vector $(\nabla\theta\circ\varphi, \int_M \theta\varrho)$ we have 
    \[
        T_{(\varphi,\lambda)}(\nabla\theta\circ\varphi,\int_M \theta\varrho) = \frac{\varrho}{\lambda} \int_M \theta\varrho - \operatorname{div}(\rho\nabla\theta)\mu  .
    \]
    Taking this expression as $\dot\varrho$ we see from the definition of $\Gdiv$ that
    \begin{equation*}
        -\operatorname{div}(\rho\nabla S) = -\operatorname{div}(\rho\nabla\theta).
    \end{equation*}
    Thus, $\nabla\theta = \nabla S$.
    We now plug this into the metric $\bGdiv$:
    \begin{align*}
        &\bGdiv[\varrho]\left(\frac{\varrho}{\lambda} \int_M \theta\varrho -  \operatorname{div}(\rho\nabla\theta)\mu,\frac{\varrho}{\lambda} \int_M \theta\varrho -  \operatorname{div}(\rho\nabla\theta)\mu \right)
        =   \\
        & \int_M \lvert \nabla\theta \rvert \varrho + 
        \frac{1}{\lambda}\int_M \theta\varrho \int_M \theta\varrho + \int_M \frac{\operatorname{div}(\rho\nabla\theta)^2}{\rho} \mu = \\
        & \Gdiv[(\varphi,\lambda)]\left( (\nabla\theta\circ\varphi, \int_M \theta\varrho), (\nabla\theta\circ\varphi, \int_M \theta\varrho) \right)
    \end{align*}
    This proves the assertion.
\end{proof}







\begin{remark}\label{rem:lagrange_vs_hamiltonian}
The small conical extension $\DiffR$ with metric $\G$ (see \eqref{eq:new_metric_direct_product}) 
can be viewed as the ``common ground" for the Lagrangian and Hamiltonian extensions in constructions of an unbalanced optimal transport.

Indeed, the Hamiltonian       
$H(\varrho,\theta) = \frac{1}{2}\int_M \left( \lvert \nabla\theta\rvert^2 + \theta^2 \right)\varrho $ (see \eqref{eq:WFR-hamilt}) expressing the metric $\bGlarge$
on $\Vol$ in the dual variables 
is the sum of two terms. The first one corresponds to the Wasserstein metric, while the second 
term $\int_M \theta^2 \varrho = \int_M(\nu/\varrho)^2\varrho$ for the density $\nu:=\theta\varrho$
represents the Fisher-Rao metric. Hence the name of the WFR metric for the semi-direct generalization of UOT developed in \cite{chizat2018unbalanced, gallouet2021regularity, vialard2017diffeomorphisms}.

On the other hand, the metric $ \Gdiv$ on $\Vol$ with an extra divergence term (see \eqref{eq:trueWFR}) also has a WFR form, although not in the Hamiltonian, but in the Lagrangian setting: the first term is the Wasserstein metric, while the second, divergence term is  the degenerate ${\dot H}^1$
contribution giving the Fisher-Rao metric on $\Vol$. 
\end{remark}



\subsection{Conical Fisher--Rao metrics}

Consider the group $\Diff$ equipped with a $\dot H^1$-type metric so that  its projection to the  space $\Dens(M)$ of normalized densities is equipped with the Fisher-Rao metric. 
It also admits the conical extension with the projection $\DiffR$ on $\Vol$. 

Note that since the Fisher-Rao metric on $\Dens(M)$ is spherical, its conical extension to $\Vol\supset\Dens(M)$ is an (infinite-dimensional) positive quadrant of the (pre-Hilbert) space of highest-degree forms naturally equipped with the {\it flat} $L^2$-metric. The positive quadrant is formed by all volume forms
on the manifold.
The projection $\Diff\to\Dens(M)$ from diffeomorphisms to volume forms is known to be a Riemannian submersion
\cite{khesin2013geometry}, and it remains a Riemannian submersion for its conical extension $\DiffR\to\Vol$.

\begin{proposition}
    The space $\DiffR$ equipped with a conical $\dot H^1$-type metric has non-positive sectional curvatures.
\end{proposition}

\begin{proof}
    Indeed, under a Riemannian submersion the sectional curvature cannot decrease \cite{oneilsemi}. 
    Since the base manifold of all volume forms is flat (i.e. its sectional curvatures all vanish) under the projection, the sectional curvatures of the space $\DiffR$ must be negative or equal to zero.
\end{proof}

This conical extension also admits a finite-dimensional version, extending the one in \cite{modin2016geometry}.
Indeed, the finite-dimensional submanifold $\GL(n)\times \mathbb R_+\subset \DiffRn\times \mathbb R_+$ 
is totally  geodesic and according to Lemma \ref{lem:total-geodesic} the corresponding projection to $\Vol$
is totally geodesic as well. One can expect similar matrix decompositions coming from this Riemannian submersion, extending those in \cite{modin2016geometry}.


\appendix
\section{General form of the geodesic equations for a conical extension}\label{appendix}

\begin{theorem}\label{thm:geodesics_of_cone}
       The geodesic equations for the cone $Q\times\mathbb R_+$ formed from a Riemannian manifold $(Q,g)$ with the metric $r^2g+dr^2 $ are:
            \[\begin{cases}
                    \displaystyle  \nabla_{\dot q}\dot q + \frac{2}{\alpha}\dot\alpha\dot q=0,\\[0.5em]
                    \displaystyle  \ddot\alpha - g(\dot q,\dot q) \alpha=0,
            \end{cases}\]
        for a geodesic $\gamma = (q, \alpha) \in Q\times\mathbb R_+$.
    \end{theorem}
    \begin{proof}
For the cone $Q\times\mathbb R_+$ consider the two projections:
            \[
            \begin{tikzcd}
                Q\times\mathbb R_+ \arrow[r,"\pi"] \arrow[d,swap,"\sigma"] & \mathbb R_+\\
                Q
            \end{tikzcd}
            \]
This cone can be viewed as a warped product $\mathbb R_+ \times_f Q$ for  $f\colon\mathbb R_+ \to\mathbb R$ defined by $r\mapsto r$ and the metric defined for $v\in T_{(r,q)}\mathbb R_+\times Q$ by:
            \[\langle v,v\rangle_{(r,q)} = d\pi(v)^2 + f(r)^2\,g_q(d\sigma(v),d\sigma(v)).\]
        The geodesic equations for such a warped product is given for a geodesic $\gamma = (\alpha,q)\in\mathbb R_+\times_f Q$ by
            \[\begin{cases}
                    \displaystyle  \nabla_{\dot\alpha}\dot\alpha - g(\dot q,\dot q)(f\circ\alpha)\nabla f=0,\\[0.5em]
                    \displaystyle  \nabla_{\dot q}\dot q + \frac{2}{f\circ\alpha}\frac{d(f\circ\alpha)}{dt}\dot q=0\,,
            \end{cases}\]
            see  \cite{oneilsemi}. 
        With our setup, $\nabla f = 1$ (with the standard metric on $\mathbb R_+$), and $f\circ r = r$, and the result follows.
    \end{proof}

\begin{remark}
In the present paper we apply the corresponding conical one-dimensional extension $r^2 g(v,v)+dr^2$
to the group of diffeomorphisms and the space of normalized densities, where $g(v,v)$  is, respectively, the $L^2$-metric on ${\rm Diff}(M)$ and the Wasserstein metric on ${\rm Dens}(M)$.  

    For an arbitrary $p\in \mathbb R$, the geodesic equations for a geodesic $\gamma = (q, \alpha)$ on the cone $Q\times\mathbb R_+$ with the metric $r^{2p}g(v,v)+dr^2 $ assume the form:
        \[\begin{cases}
           \displaystyle  \nabla_{\dot q}\dot q + \frac{2p}{\alpha}\dot\alpha\dot q=0,\\[0.5em]
           \displaystyle  \ddot \alpha -  pr^{p-1}g(\dot q,\dot q) \alpha^p=0,
        \end{cases}\]
    of which Theorem~\ref{thm:geodesics_of_cone} is the special case of $p=1$, while $p=0$ corresponds to
    the direct product metric on the  cylinder $Q\times \mathbb R$.
One can also consider the one-parameter extensions $r^{2p} g+dr^2$ in infinite dimensions as well. Other hyperbolic and parabolic-type metrics for negative and positive values of $p$ might be useful in  problems of optimal transport whenever it is convenient to tune the mass balance. 
\end{remark}


 

\bibliographystyle{amsplainnat} 
\bibliography{references}

\end{document}